\documentclass[12pt]{amsart}
\usepackage[
text={440pt,575pt},
headheight=9pt,
centering
]{geometry}
\usepackage{amssymb,amsmath,amscd,enumerate,verbatim,fullpage,mathtools}
\usepackage{mathptmx,bm}
\usepackage[colorlinks=true,linkcolor=blue,citecolor=blue]{hyperref}
\usepackage[all]{xy}

\usepackage[capitalise]{cleveref}

\numberwithin{equation}{section}

\newtheorem{thm}{Theorem}[section]
\newtheorem{lem}[thm]{Lemma}
\newtheorem{conj}[thm]{Conjecture}
\newtheorem{cor}[thm]{Corollary}

\newtheorem{problem}[thm]{Problem}
\Crefname{lem}{Lemma}{Lemmas}
\Crefname{thm}{Theorem}{Theorems}

\theoremstyle{definition}
\newtheorem{defn}[thm]{Definition}
\newtheorem{ex}[thm]{Example}
\newtheorem{rem}[thm]{Remark}

\Crefname{defn}{Definition}{Definitions}

\newtheorem*{thm*}{Theorem}

\DeclareMathOperator{\Tor}{Tor}
\DeclareMathOperator{\reg}{reg}
\DeclareMathOperator{\depth}{depth}

\DeclareMathOperator{\Img}{Im}
\DeclareMathOperator{\Inc}{Inc}

\DeclareMathOperator{\ind}{ind}

\DeclareMathOperator{\Ker}{Ker}
\DeclareMathOperator{\Lcm}{Lcm}
\DeclareMathOperator{\lcm}{lcm}

\DeclareMathOperator{\maxsupp}{Msupp}
\DeclareMathOperator{\pd}{pd}
\DeclareMathOperator{\supp}{supp}

\DeclareMathOperator{\Sym}{Sym}

\newcommand{\Z}{{\mathbb Z}}
\newcommand{\N}{{\mathbb N}}

\def\Icc{{\mathcal I}}

\newcommand\bsa{{\boldsymbol a}}

\newcommand\defas{\coloneqq}

\newcommand{\wti}{\widetilde}

\begin{document}
	
	\title{On regularity and projective dimension of invariant chains of monomial ideals}
	
	\author{Dinh Van Le}
	\address{Institut f\"ur Mathematik, Universit\"at Osnabr\"uck, 49069 Osnabr\"uck, Germany}
	\email{dlevan@uos.de}
	
	\author{Hop D. Nguyen}
	\address{Institute of Mathematics, Vietnam Academy of Science and Technology, 18 Hoang Quoc Viet, 10307 Hanoi, Vietnam}
	\email{ngdhop@gmail.com}
	
	\begin{abstract}
		Ideals in infinite-dimensional polynomial rings that are invariant under the action of the monoid of increasing functions have been extensively studied recently. Of particular interest is the asymptotic behavior of truncations of such an ideal in finite-dimensional polynomial subrings. It has been conjectured that the Castelnuovo--Mumford regularity and projective dimension are eventual linear functions along such truncations. In the present paper we provide evidence for these conjectures. We show that for monomial ideals the projective dimension is eventually linear, while the regularity is asymptotically linear.
	\end{abstract}
	
	\makeatletter
	\@namedef{subjclassname@2020}{%
		\textup{2020} Mathematics Subject Classification}
	\makeatother
	
	
	\maketitle
	
	\section{Introduction}
	
	Ideals in infinite-dimensional polynomial rings that are invariant under the action of the infinite symmetric group $\Sym$ or the monoid of increasing functions $\Inc$ 
	have recently been intensively studied, with motivations from group theory \cite{Co67}, algebraic statistics \cite{HS12}, and representation theory \cite{CEF15}; see \cite{Dr14} for a nice survey. A fascinating research direction in this theory is to explore whether existing results related to ideals in Noetherian polynomial rings can be extended to the infinite-dimensional case. Examples of successful extensions include equivariant Hilbert's basis theorem \cite{AH07,Co67,Co87,HS12,NR19}, equivariant Hilbert--Serre theorem \cite{KLS,Na,NR17}, equivariant Buchberger algorithm \cite{Co87,HKL}, and equivariant Hochster's formula \cite{MR20}. See, e.g., also \cite{DEF,NS,SS16,SaSn17} for related results. It should be noted that these results have inspired further work in combinatorics \cite{LR20} and polyhedral geometry \cite{KLR,LR21}.
	
	One challenging open problem in the mentioned research direction is to extend the classical Hilbert's syzygy theorem to the equivariant setting. More specifically, let $R=K[x_i\mid i\in\N]$ be a polynomial ring in infinitely many variables over a field $K$ and let $I\subseteq R$ be a $\Pi$-invariant ideal (with $\Pi=\Sym$ or $\Pi=\Inc$) that is equivariantly finitely generated, i.e., $I$ is generated by finitely many $\Pi$-orbits. It is of great interest to understand the minimal free resolution of $I$ over $R$, and in particular, to understand how the finiteness property of $I$ is reflected in that resolution.
	
	An approach to this problem is as follows. One considers a 
	$\Pi$-invariant chain $\Icc=(I_n)_{n\ge 1}$ with $I_n$ an ideal in the Noetherian polynomial ring $R_n=K[x_1,\dots,x_n]$ such that $I=\bigcup_{n\ge 1}I_n$ (e.g., one can take the \emph{saturated} chain $\overline{\Icc}=(\bar I_n)_{n\ge 1}$ with $\bar I_n=I\cap R_n$ for $n\ge1$). The ideal $I$ may be regarded as a \emph{global} object and the chain $\Icc$ a \emph{local} one. By studying free resolutions of the local ideals $I_n$ one has an opportunity to understand a free resolution of the global ideal $I$. This approach is undertaken in \cite{NR19}, where the chain $\Icc$ is viewed as an FI- or OI-module and a projective resolution of this module is obtained with interesting finiteness properties (see \cite[Theorems 7.1 and 7.7]{NR19}). 
	
	Numerical aspects of free resolutions of the local ideals $I_n$, namely their Castelnuovo--Mumford regularity and projective dimension, are further studied in \cite{LNNR, LNNR2}, where the following expectations are proposed.

	\begin{conj}
		\label{conj}
		Let $\Pi=\Sym$ or $\Pi=\Inc$ and
		let $\Icc=(I_n)_{n\ge 1}$ be a $\Pi$-invariant chain of proper homogeneous ideals. Then the following hold.
		\begin{enumerate}
			\item 
			\label{conj_asymptotic_pdim}
			$\pd I_n$ is eventually a linear function in $n$.
			\item 
			\label{conj_asymptotic_regularity}
			$\reg I_n$ is eventually a linear function in $n$.
		\end{enumerate}
	\end{conj}
	
	In fact, this conjecture is stated in \cite{LNNR, LNNR2} for more general polynomial rings. See also \cite{Na} for a further generalization of the conjecture to the context of FI- and OI-modules.
	
	\Cref{conj} is verified for $\Sym$-invariant chains of monomial ideals by Murai \cite{Mu}. Raicu \cite{Ra} independently also proves \Cref{conj}(ii) for such chains. Recently, Murai's results have been refined in his joint work with Raicu on equivariant Hochster's formula \cite{MR20}.
	
	In the present paper we study \Cref{conj} in the case that 
	$\Icc$ is an $\Inc$-invariant chain of monomial ideals. Our main result reads as follows.
	
	\begin{thm}
		\label{thm_main}
		Let $\Icc=(I_n)_{n\ge 1}$ be a nonzero $\Inc$-invariant chain of proper monomial ideals. Then the following hold.
		\begin{enumerate}
			\item 
			$\pd I_n$ is eventually a linear function in $n$, provided that $\Icc$ is a saturated chain.\smallskip
			\item 
			$\lim \limits_{n\to\infty}\dfrac{\reg I_n}{n}$ exists.
		\end{enumerate}
	\end{thm}
	
	The first part of this theorem generalizes Murai's result on $\Sym$-invariant chains \cite[Corollary 3.7]{Mu}, and our proof improves a key step in his argument (see \Cref{lem_betti_num_saturated2}).
	
	For the second part of \Cref{thm_main} we actually show that
	$\lim\limits_{n\to\infty}\left({\reg I_n}/{n}\right)$ has the expected value that is predicted by \cite[Conjecture 4.12]{LNNR} (see \Cref{thm_reg_asymp}). This is done by using an induction on the so-called 
	\emph{$q$-invariant} introduced in \cite{NR17}. There are two new ingredients in our proof. First, we introduce a certain chain of ideals to reduce the $q$-invariant (see \Cref{defn_colon_filtration}). This chain is somewhat similar to the one employed in \cite{LNNR,LNNR2,NR17}, but has an additional property that it is $\Inc$-invariant and thus suitable for our inductive argument. Second, we define a function $\lambda$ that records the minimal last exponent of a generator of a monomial ideal (see \Cref{sec_weight-functions}). This function can be extended to $\Inc$-invariant chains and will be exploited to give a lower bound on regularity along those chains. Such a bound is essential for a base case in the inductive argument (see \Cref{lem_saturated_ultimate}).
	
	Most positive results so far on \Cref{conj} deal with saturated chains. It is therefore worth mentioning that the constructions employed in the proof of \Cref{thm_main}(ii) are useful for studying non-saturated chains. Note that for those chains of ideals, the initial behavior of the functions $\reg I_n$ and $\pd I_n$ can be rather wild,  even after the point where the chains stabilize (see \Cref{ex_pd-decrease}).

	Let us now describe the structure of the paper. In the next section we set up notation and review relevant facts on invariant chains and monomial ideals. The proof of \Cref{thm_main}(i) is given in \Cref{sec_pd}. In \Cref{sec_reg} we prove that $\lim\limits_{n\to\infty}\left({\reg I_n}/{n}\right)$ exists and has the expected integral value. In addition, we show that \Cref{conj}(ii) is true for any $\Inc$-invariant chains of monomial ideals up to an appropriate saturation (see \Cref{cor_saturation}). Some open problems are discussed in \Cref{sec_problem}. Finally, the Appendix provides the proofs of two technical results in \Cref{sec_reg}.

	\section{Preliminaries}
	\label{sect_notation}
	
	In this section we provide necessary notions and basic facts concerning invariant chains and monomial ideals.
	
	\subsection{Invariant chains of ideals}
	Let $\N$ denote the set of positive integers. For $n\in\N$ let $\Sym(n)$ be the symmetric group on $[n]=\{1,\dots,n\}.$ It is convenient to identify each element of $\Sym(n)$ with a permutation of $\N$ that fixes every $k\ge n+1$. In this manner, $\Sym(n)$ is exactly the stabilizer subgroup of $n+1$ in $\Sym(n+1)$. Set 
	\[
	\Sym=\bigcup_{n\ge1}\Sym(n).
	\]
	Then $\Sym$ is the group of \emph{finite permutations} of $\N$, i.e. permutations that fix all but finitely many elements of $\N$. In the literature $\Sym$ is sometimes also denoted by $\Sym(\infty)$.
	
	The \emph{monoid of strictly increasing
		maps on $\N$} is defined as
	\[
	\Inc = \{ \pi \colon \N \to \N \mid  \pi(n)<\pi(n+1) \text{ for all } n\geq 1\}.
	\]
	For $m\le n$ set
	\[
	\Inc_{m,n} = \{\pi \in \Inc \mid \pi(m) \le n\}.
	\]
	
	Throughout let $R=K[x_i\mid i\in\N]$ be an infinite-dimensional polynomial ring over a field $K$. For $n\in\N$ the Noetherian polynomial ring $R_n=K[x_1,\dots,x_n]$  is regarded as a subring of $R$. Thus, one may view $R=\bigcup_{n\ge1}R_n$ as the limit of the chain of increasing subrings 
	$R_1\subset R_2\subset\cdots.$ 
	
	Let $\Pi=\Sym$ or $\Pi=\Inc$. Consider the action of $\Pi$ induced by
	\[
	\pi \cdot x_{i}=x_{\pi(i)}
	\quad\text{for every } \pi\in \Pi \text{ and } i\ge 1.
	\]
	In what follows we often write $\pi(f)$ instead of $\pi\cdot f$ whenever $\pi\in \Pi$ and $f\in R$.
	An ideal $I\subseteq R$ is \emph{$\Pi$-invariant} if
	\[
	\Pi(f)=\{\pi(f)\mid \pi\in \Pi\}\subseteq I
	\quad\text{for every } f\in I.
	\]
	A fundamental result of Cohen \cite{Co67,Co87} (see also \cite{AH07,HS12}) says that the ring $R$ is \emph{$\Pi$-equivariantly Noetherian}, in the sense that every $\Pi$-invariant ideal $I\subseteq R$ is generated by finitely many orbits, i.e. there exist $f_1,\dots,f_m\in R$ such that
	$
	I=\langle\Pi(f_1),\dots,\Pi(f_m)\rangle.
	$
	In this form, Cohen's result is also known as an equivariant Hilbert's basis theorem, and like the classical result of Hilbert,  it also has an equivalent version in terms of chains of ideals, stating that every $\Pi$-invariant chain of ideals stabilizes (see \cite{HS12} and also \cite{KLR}). Here, a chain of ideals 
	$\Icc=(I_n)_{n\ge 1}$ with $I_n\subseteq R_n$ for $n\ge 1$ is called \emph{$\Pi$-invariant} if
	\begin{equation}
		\label{eq-invariant}
		\langle\Pi_{m,n}(I_m)\rangle_{R_n}\subseteq I_n
		\quad\text{for all }\ n\ge m\ge 1,
	\end{equation}
	where $\langle\Pi_{m,n}(I_m)\rangle_{R_n}$ denotes the ideal in $R_n$ generated by $\Pi_{m,n}(I_m)$ with
	\[
	\Pi_{m,n}=
	\begin{cases}
		\Sym(n)&\text{if }\ \Pi=\Sym,\\
		\Inc_{m,n}&\text{if }\ \Pi=\Inc.
	\end{cases}
	\]
	We say that this chain \emph{stabilizes} if there exists $r\in\N$ such that the inclusion in \eqref{eq-invariant} becomes an equality for all $n\ge m \ge r$. The smallest number $r$ with this property is called the \emph{stability index} of $\Icc$ and denoted by $\ind(\Icc)$. 
	
	The two versions of Cohen's result indicate that each $\Pi$-invariant ideal $I\subseteq R$ (viewed as a global object) is intimately related to any $\Pi$-invariant chain of ideals $\Icc=(I_n)_{n\ge 1}$ (viewed as a local object) with limit $I$, i.e. $I=\bigcup_{n\ge 1}I_n$. This kind of local-global relations has been further explored in the context of polyhedral geometry in \cite{KLR,LR21}. 
	
	Among all $\Pi$-invariant chains with the same limit $I\subseteq R$, the \emph{saturated} chain $\overline{\Icc}=(\bar I_n)_{n\ge 1}$ defined by $\bar I_n=I\cap R_n$ for $n\ge1$ is evidently the largest one. It should be mentioned that any $\Sym$-invariant chain eventually coincides with a saturated one (see \cite[Corollary 6.1]{KLR}). However, this is not true for $\Inc$-invariant chains, and in fact, there are subtle differences between saturated chains and non-saturated ones (see, e.g. \Cref{ex_pd-decrease}). 
	
	Note that $\Sym$-invariant chains are always $\Inc$-invariant. This is because $\Inc_{m,n}(f)\subseteq\Sym(n)(f)$ for all $n\ge m$ and $f\in R_m.$ Even if one is primarily interested in $\Sym$-invariant chains, it is usually more convenient to work with the (strictly) larger class of $\Inc$-invariant chains. The reason is that this class is closed under taking initial ideals (see, e.g. \cite[Lemma 7.1]{NR17}), while the smaller class of $\Sym$-invariant chains is not (see, e.g. \cite[Example 2.2]{LNNR}).
	
	\subsection{Betti numbers of monomial ideals}
	
	Multigraded Betti numbers of monomial ideals can be computed via simplicial homology. We recall here this formula and a few related facts.
	Let $J$ be a monomial ideal in $R_n=K[x_1,\ldots,x_n]$ with minimal set of monomial generators $G(J)$. For 
	$\bsa=(a_1,\ldots,a_n)\in \Z^n_{\ge0}$
	let $x^\bsa=\prod_{i=1}^nx_i^{a_i}$. The \emph{support} of $x^\bsa$ is denoted by
	\[
	\supp(x^\bsa)=\{i\in[n]\mid a_i>0\}.
	\]
	Let $\maxsupp(x^\bsa)=\max\{i\mid i\in \supp(x^\bsa)\}$ be the maximal index in $\supp(x^\bsa)$.
	We define
	\[
	\maxsupp(J)=\max\{\maxsupp(u)\mid u\in G(J)\}.
	\]
	
	Endow $R_n$ with the standard $\Z^n$-grading.  Then $J$ admits a minimal $\Z^n$-graded free resolution over $R_n$.  For 
	$i\ge 0$ and $\bsa\in \Z^n_{\ge0}$ 
	let $\beta_{i,\bsa}(J)=\dim_K \Tor_i(J,K)_\bsa$ denote the $i$th \emph{Betti number} of $J$ in degree $\bsa$. Moreover, let $x^F=\prod_{j\in F}x_j$ for any $F\subseteq[n]$.  The \emph{upper Koszul simplicial complex} of $J$ in degree $\bsa$ is defined as
	\[
	\Delta^J_\bsa=\left\{F\subseteq [n]\mid \frac{x^{\bsa}}{x^F} \in J \right\}.
	\]
	Let $\Lcm(J)$ be the \emph{lcm lattice}  of $J$, i.e. the lattice of all least common multiple of monomials in $G(J)$. The following result is well-known (see, e.g. \cite[Theorem 1.34]{MS} and \cite[Theorem 58.8]{Pe}).
	
	\begin{lem}
		\label{lem_multigradedBetti_num}
		Let $J\subseteq R_n$ be a monomial ideal. Then for all $i\ge 0$ and $\bsa\in \Z^n_{\ge0}$ it holds that 
		\[
		\beta_{i,\bsa}(J) 
		= \dim_K \widetilde{H}_{i-1}(\Delta^J_\bsa),
		\]
		where $\widetilde{H}_{i-1}(\Delta^J_\bsa)$ denotes the $(i-1)$st reduced homology group of $\Delta^J_\bsa$ with coefficients in $K$.
		Moreover, $\beta_{i,\bsa}(J) =0$ if $x^\bsa\notin \Lcm(J)$.
	\end{lem}
	
	The functorial property of reduced homology implies the following (see \cite[Lemma 2.5]{Mu}).
	
	\begin{lem}
		\label{lem_induced_zeromap}
		Let $\Delta \subseteq \Sigma\subseteq \Gamma$ be simplicial complexes. If $\wti{H}_i(\Sigma)=0$ for some integer $i$, then the natural map 
		$\wti{H}_i(\Delta) \to \wti{H}_i(\Gamma)$ 
		induced by the inclusion $\Delta \subseteq \Gamma$ is zero.
	\end{lem}
	
	Recall that the \emph{projective dimension} $\pd J$ and the \emph{(Castelnuovo-Mumford) regularity} $\reg J$ of $J$ measure the width and the height of the Betti table of $J$, respectively. More precisely,
	\begin{align*}
		\pd J&=\max\{i\mid \beta_{i,\bsa}(J)\ne 0\text{ for some } \bsa\in \Z^n_{\ge0}\},\\
		\reg J&=\max\{|\bsa|-i\mid \beta_{i,\bsa}(J)\ne 0\},
	\end{align*}
	where $|\bsa|=\sum_{j=1}^na_j$ if $\bsa=(a_1,\ldots,a_n).$
	
	We will need the following formula for regularity. It is a consequence of \cite[Theorem 4.7]{CH+}; see \cite[Corollary 5.2]{LNNR2} for a proof.
	\begin{lem}
		\label{lem_reg_modulo_variable}
		Let $J\subseteq R_n$ be a nonzero monomial ideal. Consider a variable $x_k$ of $R_n$. Let $d\ge 0$ be an integer such that $J:x_k^d=J:x_k^{d+1}$. Then
		\[
		\max \{\reg \langle J:x_k^e,x_k\rangle\mid0\le e \le d\} 
		\le \reg J\in 
		\{\reg \langle J:x_k^e,x_k\rangle+e\mid0\le e \le d \}. 
		\]
	\end{lem}

	\subsection{Weight functions} 
	\label{sec_weight-functions}
	
	We introduce two invariants of monomial ideals that play an important role in this paper. 
	Let $J\subsetneq R_n$ be a nonzero monomial ideal with minimal set of monomial generators $G(J)$. For  
	$\bsa\in \Z^n_{\ge0}$ set $\lambda(x^\bsa)=a_k$ with $k=\maxsupp(x^\bsa)$ and 
	$w(x^\bsa)=\max\{a_i\mid i\in\supp(x^\bsa)\}$.
	Define
	\begin{align*}
		\lambda(J)&=\min\{\lambda(u)\mid u\in G(J)\},\\
		w(J)&=\min\{w(u)\mid u\in G(J)\}.
	\end{align*}
	Evidently, $\lambda(u)\le w(u)$ for all $u\in G(J)$, hence 
	$\lambda(J)\le w(J)$.
	
	\begin{ex}
		If 
		$J=\langle x_1^4x_2,\, x_1^3x_3^2\rangle\subseteq R_3$, then
		\[
		\lambda(J)=\lambda(x_1^4x_2)=1,\ \
		w(J)=w(x_1^3x_3^2)=3.
		\]
	\end{ex}

	The function $w$ is introduced in \cite{LNNR2} (in a more general context) as a measure of the growth of regularity (see \Cref{conj_reg}). The function $\lambda$, on the other hand, will be used in the next section to estimate the non-vanishing degrees of Betti numbers, resulting in lower bounds for regularity and projective dimension. For our purposes, it it convenient to extend these functions to invariant chains of monomial ideals. To this end, we need the following.
	
	\begin{lem}
		\label{lem_lambda}
		Let $\Icc=(I_n)_{n\ge 1}$ be a nonzero $\Inc$-invariant chain of  proper monomial ideals with $r=\ind(\Icc)$. Then 
		$$
		\lambda(I_n)\le \lambda(I_{n+1})\le w(I_{n+1})=w(I_r) 
		\quad \text{for all }\ n\ge r.
		$$
	\end{lem}
	
	\begin{proof}
		Assume $n\ge r$ and let $u\in G(I_{n+1})$ with 
		$\lambda(u)=\lambda(I_{n+1})$. Then there exist $v\in G(I_n)$ and $\pi\in\Inc_{n,n+1}$ such that 
		$u=\pi(v)$. Since $\lambda(u)=\lambda(\pi(v))=\lambda(v)$, one gets
		\[
		\lambda(I_n)\le \lambda(v) = \lambda(u)=\lambda(I_{n+1}).
		\]
		As noted above, the inequality 
		$\lambda(I_{n+1})\le w(I_{n+1})$ 
		follows at once from definition. The last equality is essentially shown in \cite[Lemma 3.5]{LNNR2}, but we include its proof here for convenience of the reader. Arguing similarly as for the function $\lambda$ one obtains 
		$w(I_r)\le w(I_{n+1})$. Now take $u'\in G(I_{r})$ with 
		$w(u')=w(I_r)$. Since $u'\in I_r\subseteq I_{n+1}$, there exists $v'\in G(I_{n+1})$ that divides $u'$. It follows that
		\[
		w(I_{n+1})\le w(v')\le w(u')=w(I_{r}).
		\qedhere
		\]
	\end{proof}
	
	The preceding lemma implies that the sequence 
	$(w(I_n))_{n\ge r}$ is constant, while the sequence 
	$(\lambda(I_n))_{n\ge r}$ is eventually constant. This permits the following definition.
	
	\begin{defn}
		Let $\Icc=(I_n)_{n\ge 1}$ be a nonzero $\Inc$-invariant chain of proper monomial ideals with $r=\ind(\Icc)$. Define
		\begin{align*}
			\lambda(\Icc)&=\max\{\lambda(I_n)\mid n\ge r\},\\
			w(\Icc)&=w(I_r).
		\end{align*}
	\end{defn}
	
	By \Cref{lem_lambda} it is clear that 
	$\lambda(\Icc)\le w(\Icc).$ The next result simplifies the computation of $\lambda(\Icc)$ in the case of saturated chains.
	
	\begin{lem}
		\label{lem_lambda_saturated}
		Let $\Icc=(I_n)_{n\ge 1}$ be a nonzero $\Inc$-invariant chain of proper monomial ideals with $r=\ind(\Icc)$. If $\Icc$ is saturated, then 
		\[
		\lambda(\Icc)=\lambda(I_n)
		\quad \text{for all }\ n\ge r.
		\]
	\end{lem}
	
	\begin{proof}
		By \Cref{lem_lambda}, it suffices to show that $G(I_r)\subseteq G(I_n)$ for all $n\ge r$. Let $u\in G(I_r)$. If $u\not\in G(I_n)$, there would exist $v\in I_r$ and $\pi\in\Inc_{r,n}$ such that $\pi(v)$ is a proper divisor of $u$. Since $u\in R_r$, this gives $\pi(v)\in R_r$. Hence, $\pi(v)\in I_n\cap R_r=I_r$ because $\Icc$ is saturated. But this contradicts the fact that $u\in G(I_r)$.
	\end{proof}
	
	Note that this lemma is not true without the assumption that $\Icc$ is saturated.
	
	\begin{ex}
		\label{ex_lambda}
		Let $\Icc=(I_n)_{n\ge 1}$ be an $\Inc$-invariant chain with $I_3=\langle x_1^2,\; x_2^2x_3,\; x_3^2\rangle$ and 
		$\ind(\Icc)=3.$
		Then $I_n=\langle x_i^2\mid 1\le i\le n\rangle$ for $n\ge 4$. Thus,
		$
		\lambda(\Icc)=2>1=\lambda(I_3).
		$
	\end{ex}

	We conclude this section with a simple criterion for eventual linearity of numerical functions.
	\begin{lem}
		\label{lem_sublinearity}
		Let $f:\Z_{\ge0}\rightarrow\Z_{\ge0}$ be a function and let 
		$s\in \Z_{\ge0}$. Then there exists $t\in \Z$ such that 
		$f(n)=sn+t$ for $n\gg0$ if one of the following conditions is satisfied:
		\begin{enumerate}
			\item 
			there exists $r\in \Z$ such that for all $n\gg0$, it holds that $f(n+1)\le f(n)+s$  and $f(n)\ge sn+r$;
			\item
			there exists $r\in \Z$ such that for all $n\gg0$,  it holds that $f(n+1)\ge f(n)+s$ and $f(n)\le sn+r$.
		\end{enumerate}
	\end{lem}

	\begin{proof}
		We only prove (i), since the proof of (ii) is similar. Setting $t_n=f(n)-sn$ one deduces from (i) that $r\le t_{n+1}\le t_n$ for $n\gg0$. So the sequence $(t_n)_{n\ge 1}$ must be eventually constant.
	\end{proof}

	
	\section{Projective dimension}
	\label{sec_pd}
	
	The main goal of this section is to verify \Cref{conj}(i) for saturated $\Inc$-invariant chains of monomial ideals.
	
	\begin{thm}
		\label{thm_projdim}
		Let $\Icc=(I_n)_{n\ge 1}$ be a nonzero saturated $\Inc$-invariant chain of proper monomial ideals. Then there exists a constant $d\ge1$ such that  $\pd I_n=n-d$ for all $n\gg 0$.
	\end{thm}
	
	Recall that any $\Sym$-invariant chain of ideals is also $\Inc$-invariant and eventually saturated (see \cite[Corollary 6.1]{KLR}). Thus, \Cref{thm_projdim} generalizes Murai's result on $\Sym$-invariant chains \cite[Corollary 3.7]{Mu}. To prove this theorem we follow his strategy, relating the non-vanishing of certain Betti numbers of $I_n$ and $I_{n+1}$ when $n\gg0.$ More specifically, the main ingredient in the proof is the following generalization and slight improvement of \cite[Proposition 3.6]{Mu}.

	\begin{lem}
		\label{lem_betti_num_saturated2}
		Let $\Icc=(I_n)_{n\ge 1}$ be a nonzero saturated $\Inc$-invariant chain of proper monomial ideals with $r=\ind(\Icc)$. Let 
		$\bsa=(a_1,\ldots,a_t,0,\ldots,0) \in \Z^n_{\ge0}$ with $t\ge 1$, $a_t\ge 1$, and $n \ge r$. For each $p\in \Z_{\ge0}$ denote $(\bsa,p)=(a_1,\ldots,a_t,p,0,\ldots,0) \in \Z^{n+1}_{\ge0}$. 
		Assume that $\beta_{i,\bsa}(I_n) \neq 0$ for some $i\ge 0$. Then the following hold.
		\begin{enumerate}
			\item 
			$\lambda(\Icc)\le a_t$.
			\item
			There exists an integer $p$ with $\lambda(\Icc) \le p \le a_t$ such that
			$
			\beta_{i+1,(\bsa,p)}(I_{n+1})\neq 0.
			$
		\end{enumerate}
	\end{lem}
	
	\begin{proof}   
		(i) Since $\beta_{i,\bsa}(I_n)\neq 0$, \Cref{lem_multigradedBetti_num} implies $x^\bsa \in \Lcm(I_n)$, i.e. $x^\bsa=\lcm(u_1,\dots,u_k)$ for some $u_1,\dots,u_k\in G(J)$. It follows that $a_t=\lambda(x^\bsa)=\lambda(u_i)$ for some $i\in[k]$. Hence, $$a_t\ge \lambda(I_n)=\lambda(\Icc),$$
		where the last equality is due to \Cref{lem_lambda_saturated}.
		
		(ii) Here the proof of \cite[Proposition 3.6]{Mu} carries over with some adjustments. We include the details for the sake of completeness. 
		For $p\ge0$ set
		$$
		\Delta^{(p)}=\left\{F\subseteq [t]: x_{t+1}^p\frac{x^\bsa}{x^F} \in I_{n+1}\right\}
		$$
		and denote $\Delta=\Delta^{(0)}$. Also, let 
		$
		\Sigma
		=\Delta^{[t]}\defas\Delta \cup \{F\cup \{t\}: F\in \Delta\}
		$
		be the cone over $\Delta$ with apex $t$. 
		We first prove the following:
		\begin{enumerate}[\quad \rm(a)]
			\item 
			$\Delta^{(0)}\subseteq \Delta^{(1)} \subseteq \cdots \subseteq  \Delta^{(a_t)}$;
			\item 
			$\Delta^{I_n}_\bsa=\Delta=\Delta^{(\lambda-1)}$, where 
			$\lambda=\lambda(\Icc)$;
			\item 
			$\Delta^{I_{n+1}}_{(\bsa,p)}
			=\Delta^{(p)} \cup \{\{t+1\} \cup F: F\in  \Delta^{(p-1)}\}$;
			\item 
			$\Delta^{(a_t)} \supseteq \Sigma.$
		\end{enumerate}
		
		For (a): This is immediate.
		
		For (b): We first show that $\Delta^{I_n}_\bsa=\Delta.$
		If $F\in \Delta^{I_n}_\bsa$, then 
		$F\subseteq \supp x^\bsa\subseteq [t]$. This implies the second equality in the chain
		\begin{align*}
			\Delta^{I_n}_\bsa &= \left\{F\subseteq [n]: \frac{x^\bsa}{x^F} \in I_n\right\} 
			= \left\{F\subseteq [t]: \frac{x^\bsa}{x^F} \in I_n\right\} 
			=  \left\{F\subseteq [t]: \frac{x^\bsa}{x^F} \in I_{n+1}\right\}=\Delta.
		\end{align*}
		The third equality holds since $I_n=I_{n+1}\cap R_n$, which is because $\Icc$ is saturated. To prove 
		$\Delta=\Delta^{(\lambda-1)}$ 
		let $F\in\Delta^{(\lambda-1)}$. Then 
		\[
		x_{t+1}^{\lambda-1}\frac{x^\bsa}{x^F} \in I_{n+1}.
		\]
		This means that $x_{t+1}^{\lambda-1}\frac{x^\bsa}{x^F}$ is divisible by some $u\in G(I_{n+1})$. Since $\lambda(u)\ge \lambda(I_{n+1})=\lambda$ by \Cref{lem_lambda_saturated}, it follows that $u$ divides $\frac{x^\bsa}{x^F}$. Hence, 
		$\frac{x^\bsa}{x^F}\in I_{n+1}$, and therefore
		$F\in \Delta^{(0)}=\Delta.$
		
		For (c): 
		If $F\in \Delta^{I_{n+1}}_{(\bsa,p)}$, then $F \subseteq \supp x^{(\bsa,p)} \subseteq [t+1]$. Assume that $F\subseteq [t+1]$. 
		If $t+1\notin F$, then 
		\begin{align*}
			& F\in \Delta^{I_{n+1}}_{(\bsa,p)}  
			\Leftrightarrow 
			F\subseteq [t] ~ \text{ and } ~ x_{t+1}^p\frac{x^\bsa}{x^F} \in I_{n+1} 
			\Leftrightarrow 
			F\in \Delta^{(p)}.
		\end{align*}
		On the other hand, if $t+1\in F$, then $F=\{t+1\}\cup G$ with $G\subseteq [t]$. Hence,
		\begin{align*}
			F\in \Delta^{I_{n+1}}_{(\bsa,p)} 
			& \Leftrightarrow 
			G\subseteq [t] ~ \text{ and } ~ x_{t+1}^p\frac{x^\bsa}{x^{G\cup \{t+1\}}} \in I_{n+1} \\
			&\Leftrightarrow 
			G\subseteq [t] ~ \text{ and } ~ x_{t+1}^{p-1}\frac{x^\bsa}{x^G} \in I_{n+1} \\
			&\Leftrightarrow
			G\in \Delta^{(p-1)}.
		\end{align*}
		Thus we obtain (c).
		
		For (d): 
		It suffices to show that if $F\in \Delta$ and $t\notin F$, then $F\cup \{t\} \in \Delta^{(a_t)}$. Since $F\in \Delta$, it follows from (a) that $F\subseteq [t-1]$ and 
		\begin{equation}
			\label{eq_I_n}
			\frac{x^\bsa}{x^F}=\frac{x_1^{a_1}\cdots x_{t-1}^{a_{t-1}}}{x^F}x_t^{a_t} \in I_n.
		\end{equation}
		Let $\sigma_{t-1}:\N\to\N$ be the map defined by
		\[
		\sigma_{t-1}(i)
		=
		\begin{cases}
			i&\text{if }\ 1\le i\le t-1,\\
			i+1&\text{if }\ i\ge t.
		\end{cases}
		\]
		Evidently, $\sigma_{t-1}\in\Inc_{n,n+1}$. Thus 
		$\sigma_{t-1}(I_n)\subseteq I_{n+1}$, and so \eqref{eq_I_n}
		implies that
		\[
		I_{n+1} \ni x_t^{a_t-1}\sigma_{t-1}\left(\frac{x^\bsa}{x^F}\right)= x_t^{a_t-1}\frac{x_1^{a_1}\cdots x_{t-1}^{a_{t-1}}}{x^F}x_{t+1}^{a_t}=x_{t+1}^{a_t}\frac{x^\bsa}{x^{F\cup \{t\}}}.
		\]
		Therefore $F\cup \{t\} \in \Delta^{(a_t)}$, as wanted.

		Next we prove the desired assertion on Betti numbers. 
		By \Cref{lem_multigradedBetti_num} and Claim (b), one gets 
		$$
		0 < \beta_{i,\bsa}(I_n) 
		=\dim_K \widetilde{H}_{i-1}(\Delta^{I_n}_\bsa) 
		= \dim_K \widetilde{H}_{i-1}(\Delta^{(\lambda-1)}).
		$$
		Take a non-zero element 
		$\alpha \in\widetilde{H}_{i-1}(\Delta^{(\lambda-1)})$. Consider the maps
		\[
		\widetilde{H}_{i-1}(\Delta^{(\lambda-1)}) \xrightarrow{\iota_{\lambda}} 
		\widetilde{H}_{i-1}(\Delta^{(\lambda)})  \xrightarrow{\iota_{\lambda+1}} \cdots  \xrightarrow{\iota_{a_t}} 
		\widetilde{H}_{i-1}(\Delta^{(a_t)}) 
		\]
		induced by the chain in Claim (a). Denote 
		$\alpha_{\lambda-1}=\alpha$ and 
		$\alpha_k=(\iota_k \circ \cdots \circ \iota_{\lambda})(\alpha)$ for $\lambda \le k\le a_t$.
		Note that $\iota_{a_t} \circ \cdots \circ \iota_{\lambda}$ is the natural map induced by the inclusion 
		$\Delta^{(\lambda-1)}\subseteq \Delta^{(a_t)}.$ Note also that $\widetilde{H}_{i-1}(\Sigma)=0$ since $\Sigma$ is a cone. So Claim (d) and \Cref{lem_induced_zeromap} yield  
		\[
		\iota_{a_t} \circ \cdots \circ \iota_{\lambda}=0.
		\]
		In particular, $\alpha_{a_t}=0$. Thus there exists an integer $p$ with $\lambda \le p\le a_t$ such that $\alpha_{p-1}\neq 0$ but $\alpha_p=0$.
		
		Set $\Gamma= \Delta^{(p-1)}$ and let $\Gamma^{[t+1]}$ be the cone over $\Gamma$ with apex $t+1$. By Claim (c),
		$$
		\Delta^{I_{n+1}}_{(\bsa,p)}
		=\Delta^{(p)} \cup \{\{t+1\} \cup F: F\in  \Delta^{(p-1)}\}
		=\Delta^{(p)} \cup \Gamma^{[t+1]}.
		$$
		One has $\Delta^{(p)} \cap \Gamma^{[t+1]}=\Gamma$ since 
		$t+1\not\in\Delta^{(p)}$. Consider the Mayer-Vietoris exact sequence
		\[
		\cdots
		\to\widetilde{H}_i(\Delta^{I_{n+1}}_{(\bsa,p)}) 
		\xrightarrow{\delta}
		\widetilde{H}_{i-1}(\Gamma) \xrightarrow{\eta}  
		\widetilde{H}_{i-1}(\Delta^{(p)})\oplus\widetilde{H}_{i-1}(\Gamma^{[t+1]})
		=\widetilde{H}_{i-1}(\Delta^{(p)})
		\to
		\cdots,
		\]
		where $\widetilde{H}_{i-1}(\Gamma^{[t+1]})=0$ because 
		$\Gamma^{[t+1]}$ is a cone.
		One can check that $\eta(\alpha_{p-1})=\pm\alpha_p=0$. Hence, 
		$0\ne \alpha_{p-1}\in \Ker \eta=\Img \delta$. 
		It follows that 
		$\widetilde{H}_i(\Delta^{I_{n+1}}_{(\bsa,p)}) \ne0$, and therefore
		\[
		\beta_{i+1,(\bsa,p)}(I_{n+1})=
		\dim_K \widetilde{H}_i(\Delta^{I_{n+1}}_{(\bsa,p)})\neq 0. \qedhere
		\]
	\end{proof}

	Now we are ready to prove \Cref{thm_projdim}.
	
	\begin{proof}[Proof of Theorem \ref{thm_projdim}]
		Note that $\pd I_n\le n-1$ for all $n\ge 1$ by Hilbert's syzygy theorem. On the other hand, it follows from \Cref{lem_betti_num_saturated2}(ii) that
		\[
		\pd I_{n+1} \ge \pd I_n+1
		\quad\text{for all }\ n\ge r.
		\]
		Now the proof is complete by using \Cref{lem_sublinearity}.
	\end{proof}

	\begin{rem}
		Keep the assumption of \Cref{thm_projdim}.
		\begin{enumerate}
			\item
			It is natural to ask when $\pd I_n$ becomes a linear function, or slightly less specifically, can one find a bound $s$ such that
			\[
			\pd I_{n+1} = \pd I_n+1
			\quad\text{for all }\ n\ge s?
			\]
			If $\Icc$ is a $\Sym$-invariant chain of monomial ideals, then $s$ can be simply chosen to be the stability index 
			$\ind(\Icc)$, as shown in \cite[Corollary 3.7]{Mu}. However, this bound is not true for $\Inc$-invariant chains; see \Cref{ex_pd-decrease}.
			\item
			By the Auslander--Buchsbaum formula, the constant $d$ in \Cref{thm_projdim} can be expressed as 
			$d=1+\lim\limits_{n\to \infty} \depth(R_n/I_n).$
			It would be interesting to have a combinatorial interpretation of $d$ in terms of the chain $\Icc$.
		\end{enumerate}
	\end{rem}
	
	To close this section we note that the proofs of \Cref{thm_projdim,lem_betti_num_saturated2} 
	cannot be directly extended to arbitrary non-saturated chains. This can be seen by the following example, where the projective dimension and regularity may decrease after the stability index. 
	
	\begin{ex}
		\label{ex_pd-decrease}
		Consider the $\Inc$-invariant chain 
		$\Icc=(I_n)_{n\ge 1}$ with $\ind(\Icc)=r\ge 6$ and 
		\[
		I_r=\langle x_1,x_3,x_4,\ldots,x_{r-1}\rangle \langle x_1x_3x_4\cdots x_{r-1}\rangle +\langle x_2x_3,x_{r-1}x_r\rangle.
		\]
		Then it holds that
		\begin{align*}
			\pd I_r &= r-1,\ \ \pd I_{r+1}=3,\  \ \pd I_{r+2}=5.
		\end{align*}
		In fact, we see that $\depth(R_r/I_r)=0$ since
		$x_1x_3x_4\cdots x_{r-1}$ belongs to the socle of $R_r/I_r$. So by the Auslander--Buchsbaum formula, $\pd I_r =\pd(R_r/I_r)-1= r-1$. 
		
		In order to show $\pd I_{r+1}=3$, we first check that
		$$
		\Inc_{r,r+1}(x_1x_3x_4\cdots x_{r-1}) 
		\subseteq \langle x_3x_4, x_{r-1}x_r\rangle.
		$$
		Indeed, for any $\pi \in \Inc_{r,r+1}$ one has $4\le \pi(4)\le 5$ since $\pi: \N \to \N$ is increasing and $\pi(r)\le r+1$. If $\pi(4)=4$, then necessarily $\pi(3)=3$, so $\pi(x_1x_3x_4\cdots x_{r-1})$ is a multiple of $x_3x_4$. If $\pi(4)=5$, then $\pi(j)=j+1$ for $4\le j\le r$, and so $\pi(x_1x_3x_4\cdots x_{r-1})$ is a multiple of $\pi(x_{r-2}x_{r-1})=x_{r-1}x_r$. This yields the above containment.
		It follows that
		$$I_{r+1}
		=\langle\Inc_{r,r+1}(I_r)\rangle
		= \langle\Inc_{r,r+1}(x_2x_3,x_{r-1}x_r)\rangle
		=\langle x_2x_3,x_2x_4,x_3x_4\rangle
		+\langle x_{r-1}x_r,x_{r-1}x_{r+1},x_rx_{r+1}\rangle.
		$$ 
		Hence,
		\[
		\pd I_{r+1}=\pd \langle x_2x_3,x_2x_4,x_3x_4\rangle
		+\pd \langle x_{r-1}x_r,x_{r-1}x_{r+1},x_rx_{r+1}\rangle+1=3.
		\]
		
		Similarly, $\pd I_{r+2}=5$. 
		Thus, although the chain $\Icc$ has stability index $r$, the sequence $(\pd I_n)_n$ may fluctuate for a while from $r$ before the stability pattern sets in. 
		
		It is also not hard to show that 
		$\reg I_r =r-1 > \reg I_{r+1}=\reg I_{r+2}=3$. Note that this cannot happen for saturated chains, since for such a chain the sequence $(\reg I_n)_n$ is always non-decreasing after the stability index (see \cite[Proof of Proposition 4.14]{LNNR2}).
	\end{ex}
	
	
	\section{Regularity}
	\label{sec_reg}
	
	For chains of monomial ideals \Cref{conj}(ii) has the following more specific form (see \cite[Conjecture 4.12]{LNNR2}).
	
	\begin{conj}
		\label{conj_reg}
		Let $\Icc=(I_n)_{n\ge 1}$ be a nonzero $\Inc$-invariant chain of proper monomial ideals. Then there exists a constant $D$ such that
		\[
		\reg I_n=(w(\Icc)-1)n+D
		\quad\text{for all }\ n\gg0.
		\]
	\end{conj}

	The main goal of this section is to prove the following evidence for the previous conjecture. We also verify this conjecture for a special class of chains of monomial ideals (see \Cref{lem_saturated_ultimate}) and show that it is true up to a certain saturation (see \Cref{cor_saturation}).
	
	\begin{thm}
		\label{thm_reg_asymp}
		Let $\Icc=(I_n)_{n\ge 1}$ be a nonzero $\Inc$-invariant chain of proper monomial ideals. Then
		\[
		\lim_{n\to\infty}\frac{\reg I_n}{n}=w(\Icc)-1.
		\]
	\end{thm}

	We prove this result using induction on the \emph{$q$-invariant} of $\Icc$, defined as follows.

	\begin{defn}
		\label{defn_delta_alpha}
		Let $J\subseteq R_n$ be a nonzero monomial ideal with 
		$\maxsupp(J)=p$. Let $\delta(J)$ denote the maximal degree of a minimal monomial generator of $J$. Then the number
		\[
		{q}(J) 
		= 
		\sum_{j=0}^{\delta(J)}\dim_K \Big(\frac{R_p}{J\cap R_p}\Big)_j
		\]
		is called the \emph{$q$-invariant} of $J$.
		
		For an $\Inc$-invariant chain of nonzero monomial ideals $\Icc=(I_n)_{n\ge 1}$ with $r=\ind(\Icc)$, we set $q(\Icc)=q(I_r)$ and call it the \emph{$q$-invariant} of $\Icc$.
	\end{defn}
	
	The $q$-invariant is introduced in \cite{NR17} (in a slightly different form). It is very useful for inductive arguments on $\Inc$-invariant chains of monomial ideals, as illustrated in \cite{LNNR,LNNR2,NR17}.

	\begin{ex}\hfill
		\begin{enumerate}
			\item 
			For an $\Inc$-invariant chain of nonzero monomial ideals $\Icc=(I_n)_{n\ge 1}$, $q(\Icc)=0$ if and only if $I_n=R_n$ for some $n\ge 1.$
			\item
			Let $J=\langle x_1^2\rangle \subseteq R_2$. Then 
			$\maxsupp(J)=1$, $\delta(J)=2$ and
			\[
			{q}(J) 
			= \sum_{j=0}^{2}\dim_K \Big(\frac{R_1}{J\cap R_1}\Big)_j
			=\sum_{j=0}^{2}\dim_K \Big(\frac{K[x_1]}{\langle x_1^2\rangle}\Big)_j
			=2.
			\]
		\end{enumerate}
	\end{ex}
	
	The following construction, inspired by an analogous construction in \cite[Lemma 6.10]{NR17}, will be used to decrease the $q$-invariant.
	
	\begin{defn}
		\label{defn_colon_filtration}
		Let $e\ge 0$ be an integer and $\Icc=(I_n)_{n\ge 1}$ a nonzero $\Inc$-invariant chain of monomial ideals with $\ind(\Icc)=r$. Denote $p=\maxsupp(I_r)$.  We define a chain $\Icc_e=(I_{e,n})_{n\ge 1}$ as follows: 
		$$
		I_{e,n}= \begin{cases}
			0 &\text{if $n \le r$},\\
			\langle \left(I_n:x^e_{n-r+p}\right) \cap R_{n-r+p-1} \rangle_{R_n} &\text{if $n\ge r+1$}.
		\end{cases}
		$$ 
	\end{defn}
	\begin{rem}
		Keep the assumption of \Cref{defn_colon_filtration}.
		\begin{enumerate}
			\item 
			It is evident that $\maxsupp(I_n)\le n-r+p$ for all $n\ge r$ since $I_n=\langle\Inc_{r,n}(I_r)\rangle_{R_n}$, and one can show that equality occurs if the chain $\Icc$ is saturated.
			\item
			The chain $\Icc_e$ in \Cref{defn_colon_filtration} is different from the one considered in \cite[Lemma 6.10]{NR17}. The latter is defined by taking colon ideals by fixed variables, while in \Cref{defn_colon_filtration} we essentially take colon ideals by last variables in the supports. 
			In this paper it is more convenient to work with the chain in \Cref{defn_colon_filtration} because it is also $\Inc$-invariant (see \Cref{prop_colon_filtration} below). The chain in \cite{NR17}, on the other hand, is only invariant with respect to a certain submonoid of $\Inc$.
		\end{enumerate}
	\end{rem}
	
	Basic properties of the chain $\Icc_e$ are collected in the next result. We postpone its technical proof to the Appendix. See \cite[Lemmas 5.3 and 5.10]{LNNR2} for analogous properties of the chain considered in \cite[Lemma 6.10]{NR17}.
	
	\begin{lem}
		\label{prop_colon_filtration}
		Let $e\ge 0$ be an integer and $\Icc=(I_n)_{n\ge 1}$ a nonzero $\Inc$-invariant chain of monomial ideals with $\ind(\Icc)=r$. Employ the notation of \Cref{defn_colon_filtration}. Then the following hold.
		\begin{enumerate}
			\item 
			$\Icc_e$ is a nonzero $\Inc$-invariant chain of monomial ideals with $\ind (\Icc_e) = r+1$. 
			\item 
			One has
			$
			{q} (\Icc_e)\le {q}(\Icc),
			$
			and equality occurs if and only if $I_{e,r+1}=\langle I_r \rangle_{R_{r+1}}$.
			\item 
			One has $w(\Icc_e) \le w(\Icc)$, with equality if 
			$0\le e \le w(\Icc)-1$.
			\item 
			It holds that $\reg I_n\ge \reg I_{e,n}$ for all $n\ge r+1$.
		\end{enumerate} 
	\end{lem}

	To prove \Cref{thm_reg_asymp} by induction we replace the chain $\Icc$ with $\Icc_e$, employing the inequality $q(\Icc_e)\le q(\Icc)$ in \Cref{prop_colon_filtration}(ii). For the induction step it is thus necessary to distinguish the case $q(\Icc)= q(\Icc_e)$. In fact, one only needs
	to consider this case when $e=w(\Icc)-1$. This leads to the following notions.
	Recall that $\lambda(\Icc)\le w(\Icc)$ by \Cref{lem_lambda}.
	
	\begin{defn}
		\label{defn_lambda-max}
		A nonzero $\Inc$-invariant chain $\Icc=(I_n)_{n\ge 1}$ of proper monomial ideals is called \emph{$\lambda$-maximal} if 
		$\lambda(\Icc)= w(\Icc)$.
	\end{defn}
	
	\begin{defn}
		Let $\Icc=(I_n)_{n\ge 1}$ be an $\Inc$-invariant chain of ideals with $r=\ind(\Icc)$ and let $p=\maxsupp(I_r)$. We say that $\Icc$ is \emph{quasi-saturated} if
		$
		I_{n+1}\cap R_{n-r+p}=I_n\cap R_{n-r+p}
		\text{ for all } n\ge r.
		$
	\end{defn}
	
	\begin{rem}
		\label{rem_quasi-saturated}
		Let $\Icc=(I_n)_{n\ge 1}$ be a quasi-saturated chain with 
		$r=\ind(\Icc)$ and $p=\maxsupp(I_r)$. Denote by $I=\bigcup_{n\ge1}I_n$ the limit ideal of the chain. Let 
		$\overline{\Icc}=(\overline{I}_n)_{n\ge 1}$ 
		be the saturated chain of $I$, i.e. $\overline{I}_n=I\cap R_n$ for $n\ge 1$. Then 
		$
		\overline{I}_{n-r+p}
		=I\cap R_{n-r+p}
		=\bigcup_{m\ge n} (I_m\cap R_{n-r+p})
		=I_n\cap R_{n-r+p}
		\text{ for all } n\ge r.
		$
		This yields $I_n=\langle \overline{I}_{n-r+p} \rangle_{R_n}$ for all $n\ge r$. Thus the chain $\Icc$ is a shift of the saturated chain $\overline{\Icc}$ (up to finitely many ideals in the beginning of the chain), which justifies the name ``quasi-saturated''. For this reason,  \Cref{lem_lambda_saturated,lem_betti_num_saturated2} also hold for any quasi-saturated chain of monomial ideals.
	\end{rem}

	We need the following characterization of the equality 
	$q(\Icc)= q(\Icc_e)$.

	\begin{lem}
		\label{lem_q_w}
		Let $\Icc=(I_n)_{n\ge 1}$ be a nonzero $\Inc$-invariant chain of  proper monomial ideals with $r=\ind(\Icc)$. For $e\ge0$ consider the chain ${\Icc}_e$ as in \Cref{defn_colon_filtration}. The following are equivalent:
		\begin{enumerate}
			\item 
			${q}(\Icc)={q}({\Icc}_e)$;
			\item
			The chain $\Icc$ is quasi-saturated and $e\le \lambda(\Icc)-1$. 
		\end{enumerate}
		In particular, if $q(\Icc)= q(\Icc_e)$ with $e={w(\Icc)-1}$, then 
		$\Icc$ is quasi-saturated and $\lambda$-maximal.
	\end{lem}
	
	\begin{proof}
		(i)$\Rightarrow$(ii): Assume that $q(\Icc)= q(\Icc_e)$. Then 
		${I}_{e,r+1}=\langle I_r \rangle_{R_{r+1}}$ by \Cref{prop_colon_filtration}(ii). This yields
		${I}_{e,n+1}=\langle I_n \rangle_{R_{n+1}}$ for all $n\ge r$ since $\ind(\Icc_e)=r+1$. Note that 
		$\maxsupp(I_n)\le n-r+p$ for all $n\ge r$, 
		with $p=\maxsupp(I_r)$.
		Hence,
		$
		\langle I_n \rangle_{R_{n+1}} \subseteq 
		\langle I_{n+1} \cap R_{n-r+p}\rangle_{R_{n+1}}
		\subseteq \langle (I_{n+1}:x_{n-r+p+1}^e) \cap R_{n-r+p}\rangle_{R_{n+1}}
		={I}_{e,n+1}
		$
		for all $n\ge r$, and equalities must hold throughout. It follows that 
		\[
		I_{n}\cap R_{n-r+p}
		=\langle I_n \rangle_{R_{n+1}}\cap R_{n-r+p}
		=\langle I_{n+1} \cap R_{n-r+p}\rangle_{R_{n+1}}\cap R_{n-r+p}
		=I_{n+1}\cap R_{n-r+p}
		\]
		for all $n\ge r$, i.e. $\Icc$ is quasi-saturated.

		Next we show that $e\le \lambda(\Icc)-1$. If this were not the case, then by \Cref{lem_lambda_saturated} (see \Cref{rem_quasi-saturated}) there would exist a minimal generator $u=x_1^{a_1}\cdots x_k^{a_k}$ of $I_r$ with $1\le a_k\le e$. Note that $k\le p.$ Since $\Icc$ is quasi-saturated, one has
		$v=x_1^{a_1}\cdots x_{k-1}^{a_{k-1}}x_p^{a_k}\in I_r$, and hence $w=x_1^{a_1}\cdots x_{k-1}^{a_{k-1}}x_{p+1}^{a_k}\in I_{r+1}.$
		It follows that
		\[
		u/x_k^{a_k}
		=w/x_{p+1}^{a_{k}}
		\in \langle (I_{r+1}:x_{p+1}^{e}) \cap R_p \rangle_{R_{r+1}}
		={I}_{e,r+1}.
		\]
		On the other hand, $u/x_k^{a_k}\not\in \langle I_r \rangle_{R_{r+1}}$ since $u$ is a minimal generator of $I_r$. This, however, contradicts the fact that 
		${I}_{e,r+1}=\langle I_r \rangle_{R_{r+1}}$.

		(ii)$\Rightarrow$(i): By \Cref{prop_colon_filtration}(ii), it suffices to show that ${I}_{e,r+1}=\langle I_r \rangle_{R_{r+1}}$. Since the inclusion ``$\supseteq$'' is trivial, we only need to verify the reverse one. Let 
		$u\in G({I}_{e,r+1})$. 
		Then $u\in R_p$ and $ux_{p+1}^e\in I_{r+1}$ with $p=\maxsupp(I_r)$. The latter implies 
		$u\in I_{r+1}$ since $e\le \lambda(\Icc)-1=\lambda(I_{r+1})-1$, where the last equality is due to \Cref{lem_lambda_saturated}. Hence, $u\in I_{r+1}\cap R_{p}\subseteq I_r$ because $\Icc$ is quasi-saturated.
		
		Finally, the equivalence of (i) and (ii) together with \Cref{lem_lambda} yields the last statement.
	\end{proof}

	The next result verifies \Cref{conj_reg} for quasi-saturated chains that are $\lambda$-maximal. 
	\begin{lem}
		\label{lem_saturated_ultimate}
		Let $\Icc=(I_n)_{n\ge 1}$ be a nonzero $\Inc$-invariant chain of  proper monomial ideals. If $\Icc$ is quasi-saturated and $\lambda$-maximal, then there exists $D\in \Z$ such that
		\[
		\reg I_n= (w(\Icc)-1)n+D 
		\quad\text{for all }\ n\gg 0.
		\]
	\end{lem}
	
	\begin{proof}
		We can apply \Cref{lem_betti_num_saturated2} because the chain $\Icc$ is quasi-saturated 
		(see \Cref{rem_quasi-saturated}).
		Since $w(\Icc)=\lambda(\Icc)$,  \Cref{lem_betti_num_saturated2}(ii) yields
		\[
		\reg I_{n+1}\ge \reg I_n+ w(\Icc)-1
		\quad\text{for all }\ n\gg 0.
		\]
		On the other hand, \cite[Corollary 4.8]{LNNR2} says that there exists $D'\in \Z$ such that
		\[
		\reg I_n\le  (w(\Icc)-1)n+D'
		\quad\text{for all }\ n\gg 0.
		\]
		The desired conclusion now follows from \Cref{lem_sublinearity}.
	\end{proof}

	We are now in position to prove \Cref{thm_reg_asymp}.
	
	\begin{proof}[Proof of \Cref{thm_reg_asymp}]
		Let $w=w(\Icc)$. By \cite[Corollary 4.8]{LNNR2}, it suffices to show that there exists a constant $c$ such that
		\[
		\reg I_n\ge (w-1)n +c
		\quad \text{for all }\ n\gg0.
		\]
		We proceed by induction on ${q}(\Icc)$. If ${q}(\Icc)=0$, then $I_n=R_n$ for all $n\ge \ind(\Icc)$, and there is nothing to prove. So assume that ${q}(\Icc)>0$. 
		Set $e_0=w-1$. Recall that $q({\Icc}_{e_0})\le q(\Icc)$ by \Cref{prop_colon_filtration}(ii). We distinguish two cases:
		
		\emph{Case 1}: $q({\Icc}_{e_0})<q(\Icc).$ As $e_0<w$, one has $w({\Icc}_{e_0})=w$ by \Cref{prop_colon_filtration}(iii). So by the induction hypothesis there exists a constant $c$ such that
		\[
		\reg  I_{e_0,n}\ge (w-1)n +c
		\quad \text{for all }\ n\gg0.
		\]
		Since $\reg I_n\ge \reg I_{e_0,n}$ for all $n\ge \ind(\Icc)+1$ by  \Cref{prop_colon_filtration}(iv),
		the desired conclusion follows.
		
		\emph{Case 2}: $q(\Icc)=q({\Icc}_{e_0})$. In this case,  $\Icc$ is quasi-saturated and $\lambda$-maximal by \Cref{lem_q_w}. Therefore, \Cref{lem_saturated_ultimate} allows us to conclude the proof.
	\end{proof}
	
	Using a previous result in \cite{LNNR2}, an immediate consequence of \Cref{thm_reg_asymp} is that \Cref{conj_reg} is true for $\Inc$-invariant chains of Artinian monomial ideals.
	
	\begin{cor}
		\label{cor_artin_mono}
		If $\Icc=(I_n)_{n\ge 1}$ is an $\Inc$-invariant chain of monomial ideals with $\dim(R_n/I_n)=0$ for all $n\gg 0$, then there exists an integer $D$ such that $\reg I_n=(w(\Icc)-1)n+D$ for all $n\gg0$.
	\end{cor}
	\begin{proof}
		Since $\reg I_n$ is eventually a linear function by \cite[Corollary 6.5]{LNNR2}, the conclusion follows easily from \Cref{thm_reg_asymp}.
	\end{proof}
	
	\begin{rem}
		When $\reg I_n$ is a linear function (as in \Cref{lem_saturated_ultimate} and \Cref{cor_artin_mono}), it would be interesting to understand the constant $D$. In case $\Icc$ is a $\Sym$-invariant chain of monomial ideals with $r=\ind(\Icc)$ and $w=w(\Icc)$, Raicu \cite[Theorem 6.1]{Ra} (see also \cite[Remark 3.11]{Mu}) shows that 
		$D=\reg(I_r:(x_1\cdots x_r)^{w-1})$. 
		However, this is not true for $\Inc$-invariant chains, as illustrated by the following example. Consider the $\Inc$-invariant chain $\Icc=(I_n)_{n\ge 1}$ with 
		$I_r=\langle x_1x_r^w,x_r^s\rangle$ and $\ind(\Icc)=r$, where $r\ge 2$ and $s\ge w+1$.  Then $w(\Icc)=\lambda(\Icc)=w$, and $I_r:(x_1\cdots x_r)^{w-1}=\langle x_r\rangle$ has regularity 1. On the other hand, one can show that
		\[
		\reg I_n = (w-1)n+1+(r-1)(s-2w+1) \quad \text{for all \ $n\ge 2r-2$}.
		\]
		Thus, depending on $w$ and $s$, the constant 
		$D=1+(r-1)(s-2w+1)$
		can be either smaller or bigger than 
		$\reg(I_r:(x_1\cdots x_r)^{w-1}).$
	\end{rem}
	
	In the remainder of this section we show that \Cref{conj_reg} is true for any chain of monomial ideals up to an appropriate saturation. To this end, let us first introduce the following construction.
	
	\begin{defn}
		Let $\Icc=(I_n)_{n\ge 1}$ be an $\Inc$-invariant chain of ideals and let $m$ be a positive integer.  The \emph{$m$-saturation} of $\Icc$ is the chain $\Icc^{[m]}=(J_n)_{n\ge 1}$ determined by $J_1=\langle0\rangle$ and
		$$J_n=\Big\langle\sum_{k=2}^n x_{k}^mI_{k-1}\Big\rangle_{R_n}\quad\text{for } n\ge 2.$$
	\end{defn}
	
	As shown below, $\Icc^{[m]}$ is a saturated chain, justifying the name.  Nevertheless, one should not confuse $\Icc^{[m]}$ with the saturated chain $\overline{\Icc}$ considered so far. 
	
	\begin{ex}
		Consider again the chain $\Icc$ in \Cref{ex_lambda}. Let 
		$\overline{\Icc}=(\bar I_n)_{n\ge 1}$ and $\Icc^{[m]}=(J_n)_{n\ge 1}$ be the saturated chain and $m$-saturation of $\Icc$, respectively. Recall that 
		${I}_n=\langle x_i^2\mid 1\le i\le n\rangle$ for $n\ge 4$, hence 
		$I=\bigcup_{n\ge1}I_n=\langle x_i^2\mid  i\ge 1\rangle\subset R$. 
		This gives
		$
		\bar I_n=I\cap R_n=\langle x_i^2\mid 1\le i\le n\rangle
		\text{ for all } n\ge 1.
		$
		On the other hand,
		\begin{align*}
			J_4&=\langle x_4^mI_3\rangle_{R_4}=\langle x_1^2x_4^m,\; x_2^2x_3x_4^m,\; x_3^2x_4^m\rangle_{R_4},\\
			J_5&=\langle x_4^mI_3+x_5^mI_4\rangle_{R_5}=\langle x_1^2x_4^m,\; x_2^2x_3x_4^m,\; x_3^2x_4^m\rangle_{R_5}+\langle x_i^2x_5^m\mid 1\le i\le 4\rangle_{R_5},\text{ etc}.
		\end{align*}
	\end{ex}
	
	The next result gathers some properties of $m$-saturation. Its proof will be given in the Appendix.
	
	\begin{lem}
		\label{lem_saturated_new}
		Let $\Icc=(I_n)_{n\ge 1}$ be an $\Inc$-invariant chain of ideals. For $m\ge1$ let $\Icc^{[m]}=(J_n)_{n\ge 1}$ be the $m$-saturation of $\Icc$. Then
		$\Icc^{[m]}$ is a saturated $\Inc$-invariant chain. 
		If, moreover, $\Icc$ is a chain of monomial ideals, then the following hold.
		\begin{enumerate}
			\item
			$\lambda(\Icc^{[m]})=m$ and $w(\Icc^{[m]})=\max\{w(\Icc),m\}$. 
			\item
			For $n\ge 2$, $i\ge0$ and $\bsa =(a_1,\dots,a_{n-1})\in \Z^{n-1}_{\ge0}$ one has
			\[
			\beta_{i,(\bsa,m)}(J_{n})
			=\beta_{i,\bsa}(I_{n-1})+ \beta_{i-1,\bsa}(J_{n-1}),\quad
			\text{where }\ 
			(\bsa,m)=(a_1,\dots,a_{n-1},m).
			\]		
			\item
			$\reg J_{n}=\max\{\reg I_{n-1}+m,\; \reg J_{n-1}+m-1\}$\ for all $n\ge 2$.
		\end{enumerate}
	\end{lem}
	
	Let us now show that \Cref{conj_reg} is true for the $m$-saturation of an arbitrary $\Inc$-invariant chain of monomial ideals, given that $m$ is large enough.
	
	\begin{cor}
		\label{cor_saturation}
		Let $\Icc$ be a nonzero $\Inc$-invariant chain of proper monomial ideals. For $m\ge 1$ let $\Icc^{[m]}=(J_n)_{n\ge 1}$ be the $m$-saturation of $\Icc$. Then for $m\ge w(\Icc)$ there exists an integer $D$ such that
		\[
		\reg J_n=(m-1)n+D
		\quad\text{for all }\ n\gg0.
		\]
	\end{cor}
	
	\begin{proof}
		If $m\ge w(\Icc)$, then the chain $\Icc^{[m]}$ is saturated and $\lambda$-maximal by \Cref{lem_saturated_new}. So the result follows from \Cref{lem_saturated_ultimate}.
	\end{proof}
	
	We conclude this section with the following remark, showing in particular that the verification of \Cref{conj_reg} can be reduced to the case of saturated chains if $w(\Icc)\ge 2$.
	
	\begin{rem}
		Let $\Icc$ be a nonzero $\Inc$-invariant chain of proper monomial ideals with $w(\Icc)\ge 2$. In order to verify \Cref{conj_reg} for $\Icc$, it is enough to show that this conjecture is true for an $m$-saturation $\Icc^{[m]}$ of $\Icc$ for some $1\le m\le w(\Icc)-1$. Indeed, suppose that \Cref{conj_reg} holds for such an $m$-saturation. Then from \Cref{lem_saturated_new}(i) we get $w(\Icc^{[m]})= w(\Icc)$, hence 
		$$
		\reg J_n=\reg J_{n-1}+w(\Icc)-1> \reg J_{n-1}+m-1
		\quad\text{for all }\ n\gg0.
		$$
		So \Cref{lem_saturated_new}(iii) yields $\reg J_n=\reg I_{n-1}+m$ for all $n\gg0$, and therefore $\reg I_{n}=\reg J_n-m$ must be eventually a linear function.
	\end{rem}
	
	
	\section{Open problems}
	\label{sec_problem}
	In this section we briefly discuss two open problems that are related to \Cref{conj}.
	
	As mentioned in the introduction, \Cref{conj} is verified for $\Sym$-invariant chains of monomial ideals by Murai \cite{Mu} and Raicu \cite{Ra}. This result has been covered in their joint work on equivariant Hochster's formula \cite{MR20}, where they give a combinatorial description of the $\Sym(n)$-module structure of $\Tor_i(I_n,K)$ when $I_n\subseteq R_n$ is a $\Sym(n)$-invariant monomial ideal. It would be interesting to extend such an equivariant Hochster's formula to $\Inc$-invariant chains.
	
	\begin{problem}
		Let $\Icc=(I_n)_{n\ge 1}$ be an $\Inc$-invariant chain of monomial ideals. Give an equivariant description of $\Tor_i(I_n,K)$ for $n\gg0.$
	\end{problem}
	
	A possible approach to \Cref{conj} is to study primary decompositions of the ideals $I_n$. Experiments with Macaulay2 \cite{GS} lead to the following.
	
	\begin{conj}
		Let $\Icc=(I_n)_{n\ge 1}$ be an $\Inc$-invariant chain of monomial ideals. Then the number of $\Inc$-orbits of associated primes of $I_n$ is eventually a quasipolynomial function in $n$.
	\end{conj}
	
	This conjecture is closely related to a recent work of Draisma, Eggermont, and Farooq \cite{DEF}. There it is shown that if $\Icc=(I_n)_{n\ge 1}$ is a $\Sym$-invariant chain of ideals (in more general polynomial rings), then the number of $\Sym(n)$-orbits of minimal primes of $I_n$ is eventually a quasipolynomial in $n$. 
	
	
	\section{Appendix}
	
	We provide here the proofs of 
	\Cref{prop_colon_filtration,lem_saturated_new}. 
	For $j\ge 0$ let 
	$\sigma_{j}:\N\to\N$ be the increasing map defined by
	\[
	\sigma_{j}(i)
	=
	\begin{cases}
		i&\text{if }\ 1\le i\le j,\\
		i+1&\text{if }\ i\ge j+1.
	\end{cases}
	\]
	Obviously, $\sigma_{j}\in\Inc_{n,n+1}$ for all $n\ge 1.$
	Moreover, for any $f\in R_n$ it holds that
	\begin{equation}
		\label{eq_orbit}
		\Inc_{n,n+1}(f)
		=\bigcup_{j\ge0}\sigma_{j}(f)
		=\bigcup_{j=0}^{n}\sigma_{j}(f).
	\end{equation}

	\begin{proof}[Proof of \Cref{prop_colon_filtration}] 
		Denote $p=\maxsupp(I_r).$
		
		(i) It is easily seen that the chain $\Icc_e$ is $\Inc$-invariant. Let us show that $\ind(\Icc_e) = r+1$. Since $I_{e,n}=0$ for $n\le r$ by definition, it suffices to check that 
		$I_{e,n+1}=\langle\Inc_{n,n+1}(I_{e,n})\rangle_{R_{n+1}}$ for $n\ge r+1.$ Take $u\in G(I_{e,n+1})$ with $n\ge r+1$. Then $ux_{n-r+p+1}^e \in I_{n+1}$, and so 
		$ux_{n-r+p+1}^e=\sigma_{j}(v)z$ for some $j\ge0$, $v\in G(I_n)$, and $z\in R_{n+1}$.  
		Note that $\maxsupp(v)\le\maxsupp(I_n)\le n-r+p$. By definition of $\sigma_{j}$, it is clear that 
		$\maxsupp(v)\le\maxsupp(\sigma_{j}(v))\le\maxsupp(v)+1.$ 
		Consider the following cases:
		
		\emph{Case 1}: $\maxsupp(v)\le n-r+p-1$. Then 
		$\maxsupp(\sigma_{j}(v))\le n-r+p$. From 
		$ux_{n-r+p+1}^e=\sigma_{j}(v)z$ we deduce that 
		$\sigma_{j}(v)$ divides $u$. Since $v\in I_n \cap R_{n-r+p-1} \subseteq I_{e,n}$, we get 
		$u\in \Inc_{n,n+1}(I_{e,n})$.
		
		\emph{Case 2}: $\maxsupp(v)=\maxsupp(\sigma_{j}(v))= n-r+p.$
		Then $\sigma_{j}(v)=v$, and $v$ divides $u$. Note that
		$I_n=\langle\Inc_{n-1,n}(I_{n-1})\rangle_{R_{n}}$ since $n\ge r+1$. Hence $v=\sigma_{k}(v')$ for some $k\ge0$ and $v'\in I_{n-1}$. Evidently, $I_{n-1}\subseteq I_{e,n}$. Thus $u\in \langle\Inc_{n,n+1}(I_{e,n})\rangle_{R_{n+1}}$, because $u$ is divisible by 
		$\sigma_{k}(v')\in\Inc_{n,n+1}(I_{e,n}).$

		\emph{Case 3}:  $\maxsupp(v)= n-r+p$ and 
		$\maxsupp(\sigma_{j}(v))= n-r+p+1$. 
		This means that $j<n-r+p$. 
		We can write $v=v'x_{n-r+p}^d$ with $d\ge 1$ and 
		$\maxsupp(v')\le n-r+p-1$. Then from
		\[
		ux_{n-r+p+1}^{e}=\sigma_{j}(v)z
		=\sigma_{j}(v')x_{n-r+p+1}^{d}z
		\]
		we deduce that $d\le e$ and $\sigma_{j}(v')$ divides $u$. Hence,
		$v' \in \langle(I_n:x_{n-r+p}^e) \cap R_{n-r+p-1}\rangle_{R_{n}}=I_{e,n}$ 
		and again
		$u\in \langle\Inc_{n,n+1}(I_{e,n})\rangle_{R_{n+1}}$, as desired.
		
		(ii) We first observe the following isomorphism of graded rings
		\[
		\frac{R_p}{I_r\cap R_p}\cong
		\frac{R_{r+1}}{\langle I_r, x_{p+1},\ldots,x_{r+1}\rangle}.
		\]
		Denote $p'=\maxsupp(I_{e,r+1})$. Since 
		$I_{e,r+1}=\langle(I_{r+1}:x_{p+1}^e) \cap R_{p}\rangle_{R_{r+1}}$, it is obvious that $p'\le p.$
		Hence
		\[
		\langle I_r, x_{p+1},\ldots,x_{r+1}\rangle
		\subseteq
		\langle I_{e,r+1}, x_{p'+1},\ldots,x_{r+1}\rangle.
		\]
		Combining this with the fact that
		$\delta(I_r)\ge \delta(I_{r+1}) \ge \delta(I_{e,r+1})$ 
		we get
		\[
		{q}(\Icc_e) 
		= \sum_{j=0}^{\delta(I_{e,r+1})}\dim_K 
		\Big(\frac{R_{r+1}}{\langle I_{e,r+1}, x_{p'+1},\ldots,x_{r+1}\rangle}\Big)_j 
		\le \sum_{j=0}^{\delta(I_r)}\dim_K 
		\Big(\frac{R_{r+1}}{\langle I_r, x_{p+1},\ldots,x_{r+1}\rangle}\Big)_j
		={q}(\Icc).
		\]
		One checks easily that the equality occurs if and only if
		$I_{e,r+1}=\langle I_r \rangle_{R_{r+1}}$. 
		
		(iii) Choose $u\in G(I_r)$ with $w(u)=w(I_r)$. Since $ I_r \subseteq I_{e,r+1}$, there exists $v\in G(I_{e,r+1})$ that divides $u$. It follows that
		\[
		w(\Icc)=w(I_r)=w(u)\ge w(v)\ge w(I_{e,r+1})= w(\Icc_e).
		\]
		Let us show that the equality holds if $0\le e \le w(\Icc)-1$. Indeed, take any monomial $\tilde{v}\in G(I_{e,r+1})$. It suffices to prove that $w(\tilde{v})\ge w(\Icc)$. Since $\tilde{v}x_{p+1}^e\in  I_{r+1}$, it is divisible by $\sigma_j(\tilde{u})$ for some $j\ge0$ and $\tilde{u}\in G(I_r)$. This implies
		\[
		\max\{w(\tilde{v}),e\}=w(\tilde{v}x_{p+1}^e)
		\ge w(\sigma_j(\tilde{u}))=w(\tilde{u})
		\ge w(I_r)=w(\Icc).
		\]
		Hence, $w(\tilde{v})\ge w(\Icc)$ because $e<w(\Icc).$
		
		(iv) Denote $L=\langle I_n:x_{n-r+p}^e,x_{n-r+p}\rangle_{R_n}$. Observe that $I_n:x_{n-r+p}^d=I_n:x_{n-r+p}^{d+1}$ for $d\gg0$. Indeed, this is true for any $d$ at least the maximal exponent of $x_{n-r+p}$ in an element of $G(I_n)$. So we can apply  \Cref{lem_reg_modulo_variable} to get
		$\reg L\le \reg I_n$. By definition of $I_{e,n}$, it is clear that $L=I_{e,n}+\langle x_{n-r+p}\rangle_{R_n}$ and $x_{n-r+p}$ is a non-zero-divisor on $R_n/I_{e,n}$. Hence, 
		$\reg I_{e,n}=\reg L\le \reg I_n$. 
	\end{proof}

	Next we prove \Cref{lem_saturated_new}.

	\begin{proof}[Proof of \Cref{lem_saturated_new}]
		We first show that $\Icc^{[m]}$ is an $\Inc$-invariant chain. Since $\Icc$ is $\Inc$-invariant, it suffices to verify the following equalities for all $n\ge 2$:
		\begin{align*}
			J_{n}
			&=x_{n}^m\langle I_{n-1}\rangle_{R_n}+\langle J_{n-1}\rangle_{R_n},\\
			\langle\Inc_{n-1, n}(J_{n-1})\rangle_{R_{n}}
			&=x_{n}^m\langle\Inc_{n-2, n-1}(I_{n-2})\rangle_{R_{n}}+\langle J_{n-1}\rangle_{R_{n}}.
		\end{align*}
		The former is immediate from definition. To prove the latter we compute the orbit $\Inc_{n-1, n}(x_k^mf)$ for $f\in R_{k-1}$ with $k\le n-1.$ Using \eqref{eq_orbit} one gets
		\begin{align*}
			\Inc_{n-1, n}(x_{k}^mf)
			&=\bigcup_{j=0}^{k-1}\sigma_j(x_{k}^mf)\cup \bigcup_{j=k}^{n-1}\sigma_j(x_{k}^mf)
			=x_{k+1}^m\Inc_{k-1, k}(f)\cup\{x_{k}^mf\}.
		\end{align*}
		It follows that
		\[
		\begin{aligned}
			\langle\Inc_{n-1,n}(J_{n-1})\rangle_{R_{n}}
			&=\sum_{k=2}^{n-1}\langle\Inc_{n-1,n}(x_{k}^mI_{k-1})\rangle_{R_{n}}
			=\sum_{k=2}^{n-1}\langle x_{k+1}^m\Inc_{k-1, k}(I_{k-1})+x_{k}^mI_{k-1}\rangle_{R_{n}}\\
			&=x_{n}^m\langle\Inc_{n-2, n-1}(I_{n-2})\rangle_{R_{n}}+\sum_{k=2}^{n-1}x_{k}^m\langle\Inc_{k-2, k-1}(I_{k-2})+I_{k-1}\rangle_{R_{n}}\\
			&=x_{n}^m\langle\Inc_{n-2, n-1}(I_{n-2})\rangle_{R_{n}}+\sum_{k=2}^{n-1}x_{k}^m\langle I_{k-1}\rangle_{R_{n}}\\
			&=x_{n}^m\langle\Inc_{n-2, n-1}(I_{n-2})\rangle_{R_{n}}+\langle J_{n-1}\rangle_{R_{n}},
		\end{aligned}
		\]
		where the second-to-last equality is due to the fact that the chain $\Icc$ is $\Inc$-invariant.
		
		To see that the chain $\Icc^{[m]}$ is saturated, one only needs to note that
		\[
		J_{n+1}\cap R_n
		=(x_{n+1}^m\langle I_{n}\rangle_{R_{n+1}}+\langle J_n\rangle_{R_{n+1}})\cap R_n
		=J_n
		\quad\text{for all }\ n\ge 1.
		\]
		
		From now on, assume that $\Icc$ is a chain of monomial ideals.
		
		(i) Let $n\ge 2$. It is clear from the definition that $u\in G(J_n)$ precisely when $u=x_k^mv$ for some $v\in G(I_{k-1})$ and $2\le k\le n$. This yields the claim since
		$\lambda(x_k^mv)=m$ and $w(x_k^mv)=\max\{m,w(v)\}$.
		
		(ii) We exclude the trivial case $\bsa=\mathbf{0}$ and suppose that $\bsa=(a_1,\dots,a_t,0,\dots,0)$ with $a_t\ne0$ for some $t\in[n-1]$. Denote 
		$\Delta=\Delta^{I_{n-1}}_\bsa$, $\Gamma=\Delta^{J_{n-1}}_\bsa$, and $\Gamma^{[k]}$ the cone over $\Gamma$ with apex $k$. Moreover for each $p\ge0$, consider as in the proof of \Cref{lem_betti_num_saturated2} the simplicial complex
		\[
		\Gamma^{(p)}
		=\Big\{F\subseteq[n-1]\mid \frac{x^\bsa}{x^F}x_n^p\in J_n\Big\}
		=\Big\{F\subseteq[n-1]\mid \frac{x^\bsa}{x^F}\in (J_n:x_n^p)\cap R_{n-1}\Big\}.
		\]
		We claim that:
		\begin{enumerate}[(a)]
			\item 
			$\Gamma^{[t]}\subseteq \Delta$,
			\item $\Gamma^{(m)}=\Delta$ and $\Gamma^{(m-1)}=\Gamma$,
			\item
			$\Delta^{J_{n}}_{(\bsa,m)}=\Gamma^{(m)}\cup\{\{n\}\cup F\mid F\in \Gamma^{(m-1)}\}=\Delta\cup \Gamma^{[n]}$.
		\end{enumerate}
		
		For (a): Evidently, $J_{n-1}\subseteq I_{n-1}$, hence $\Gamma\subseteq\Delta$. So  its suffices to show that $F\cup\{t\}\in \Delta$ for any 
		$F\in\Gamma$ with $t\not\in F$. Set
		$\bsa'=(a_1,\dots,a_{t-1},0,\dots,0)$. 
		As the chain $\Icc^{[m]}$ is saturated, one has
		\[
		\frac{x^{\bsa'}}{x^F}x_t^{a_t}
		=\frac{x^\bsa}{x^F}\in J_{n-1}\cap R_t=J_t.
		\]
		Since 
		$J_t:x_t^{a_t}
		=(x_{t}^m\langle I_{t-1}\rangle_{R_{t}}+\langle J_{t-1}\rangle_{R_{t}}):x_t^{a_t}\subseteq \langle I_{t-1}\rangle_{R_{t}}$,
		this yields
		\[
		\frac{x^{\bsa'}}{x^F}\in J_t:x_t^{a_t}
		\subseteq \langle I_{t-1}\rangle_{R_{t}}\subseteq I_{n-1}.
		\]
		It follows that
		\[
		\frac{x^\bsa}{x^{F\cup\{t\}}}
		=\frac{x^{\bsa'}}{x^F}x_t^{a_t-1} \in I_{n-1},
		\]
		and hence $F\cup\{t\}\in \Delta$.
		
		For (b): One has
		$(J_n:x_n^m)\cap R_{n-1}= I_{n-1}$,
		hence $\Gamma^{(m)}=\Delta$. On the other hand, $\Gamma^{(m-1)}=\Gamma$ since 
		\[
		(J_{n}:x_{n}^{m-1})\cap R_{n-1}
		=\big(x_{n}\langle I_{n-1}\rangle_{R_{n}}+\langle J_{n-1}\rangle_{R_{n}}\big)
		\cap R_{n-1}
		=J_{n-1}.
		\]
		
		For (c): Arguing as for Claim (c) in the proof of \Cref{lem_betti_num_saturated2}, we obtain
		\[
		\Delta^{J_{n}}_{(\bsa,m)}
		=\Gamma^{(m)}\cup\{\{n\}\cup F\mid F\in \Gamma^{(m-1)}\}
		=\Delta\cup \Gamma^{[n]}.
		\]

		Let us now verify the stated formula for Betti numbers. One has
		$\Delta\cap \Gamma^{[n]}=\Gamma$ since $n\not\in \Delta$. 
		So using (c) we obtain the Mayer-Vietoris exact sequence
		\[
		\cdots
		\to \widetilde{H}_{i-1}(\Gamma)
		\xrightarrow{\delta_{i-1}=(\psi_{i-1},\varphi_{i-1})} 
		\widetilde{H}_{i-1}(\Delta)\oplus\widetilde{H}_{i-1}(\Gamma^{[n]})
		\to\widetilde{H}_{i-1}(\Delta^{J_{n}}_{(\bsa,m)}) 
		\to
		\widetilde{H}_{i-2}(\Gamma) 
		\xrightarrow{\delta_{i-2}}  
		\cdots.
		\]
		Since  $\Gamma^{[n]}$ is a cone, $\varphi_k:\widetilde{H}_{k}(\Gamma)
		\to\widetilde{H}_{k}(\Gamma^{[n]})=0$ 
		is a zero map for all $k$. On the other hand, it follows from (a) and \Cref{lem_induced_zeromap} that
		$\psi_k:\widetilde{H}_{k}(\Gamma)\to\widetilde{H}_{k}(\Delta)$ is also a zero map for all $k$.
		Therefore, 
		$\delta_k$ is a zero map for all $k$. The above exact sequence thus induces the following short exact sequence
		\[
		0\to\widetilde{H}_{i-1}(\Delta)
		\to\widetilde{H}_{i-1}(\Delta^{J_{n}}_{(\bsa,m)}) 
		\to \widetilde{H}_{i-2}(\Gamma)
		\to 0,
		\]
		which provides
		\[
		\beta_{i,(\bsa,m)}(J_{n})=
		\dim_K \widetilde{H}_{i-1}(\Delta^{J_{n}}_{(\bsa,m)})
		=\dim_K \widetilde{H}_{i-1}(\Delta)+\dim_K \widetilde{H}_{i-2}(\Gamma)
		=\beta_{i,\bsa}(I_{n-1})+ \beta_{i-1,\bsa}(J_{n-1}).
		\]
		
		(iii) 
		From (ii) it follows that 
		$\reg J_{n}\ge\max\{\reg I_{n-1}+m,\; \reg J_{n-1}+m-1\}.$
		On the other hand, using \Cref{lem_reg_modulo_variable} one obtains
		\[
		\reg J_{n}\in
		\{\reg\langle J_n:x_n^e,x_n\rangle+e\mid 0\le e\le m\}
		=\{\reg I_{n-1}+m,\; \reg J_{n-1}+e\mid 0\le e\le m-1\}.
		\]
		Therefore, it must hold that 
		$\reg J_{n}=\max\{\reg I_{n-1}+m,\; \reg J_{n-1}+m-1\}.$ 
	\end{proof}


	\section*{Acknowledgments}
	
	We would like to thank the referee for useful suggestions.
	This work is partially supported by the Simons Foundation Targeted Grant for the Institute of
	Mathematics - VAST (Award number: 558672), and by the Vietnam Academy of Science and Technology (grants CSCL01.01/22-23 and NCXS02.01/22-23). Parts of this work were carried out during a stay
	of the second author at the Vietnam Institute for Advanced Study in Mathematics (VIASM). He would like to thank VIASM for its hospitality and generous support. He also thanks the participants of the MFO-RIMS Tandem Workshop on ``Symmetries on Polynomial Ideals and Varieties'' for inspiring  discussions related to this work.

\end{document}